\documentclass[11pt]{amsart}
\usepackage{amsfonts,latexsym,amsthm,amssymb,amsmath,amscd,euscript,tikz, tikz-cd}
\usepackage[alphabetic, msc-links, bibtex-style, nobysame]{amsrefs}
\BibSpec{collection.article}{%
    +{}  {\PrintAuthors}                {author}
    +{.} { }                            {title}
    +{.} { }                            {part}
    +{:} { }                            {subtitle}
    +{.} { \PrintContributions}         {contribution}
    +{.} { \PrintConference}            {conference}
    +{.} { \PrintBook}                  {book}
    +{.} { In: }                         {booktitle}
    +{,} { pages~}                      {pages}
    +{.} { }                            {publisher}
    +{,} { }                            {address}
    +{,} { \PrintDateB}                 {date}
    +{.} { \PrintTranslation}           {translation}
    +{.} { Reprinted in \PrintReprint}  {reprint}
    +{.} { }                            {note}
    +{.} {}                             {transition}
}
\usepackage{stackengine}
\usepackage{framed}
\usepackage{xfrac}
\usepackage[makeroom]{cancel}
\usepackage{faktor}
\usepackage{braket}
\usepackage{pgf,tikz,pgfplots}
\usepgfplotslibrary{groupplots}
\pgfplotsset{compat=1.18}
\usepgfplotslibrary{fillbetween} 
\usepackage{mathrsfs}
\usetikzlibrary{arrows}
\definecolor{wrwrwr}{rgb}{0.3803921568627451,0.3803921568627451,0.3803921568627451}
\definecolor{rvwvcq}{rgb}{0.08235294117647059,0.396078431372549,0.7529411764705882}
\definecolor{mblue}{rgb}{0.2, 0.3, 0.8}
\definecolor{morange}{rgb}{1, 0.5, 0}
\definecolor{mgreen}{rgb}{0.1, 0.4, 0.2}
\definecolor{mred}{rgb}{0.5, 0, 0}
\definecolor{ForestGreen}{RGB}{34,139,34}
\usepackage{float}

\numberwithin{equation}{section}

\usepackage{lmodern}
\usepackage{alphabeta}
\usepackage{supertabular}
\usepackage{amssymb}
\usepackage{enumerate}
\usepackage{stmaryrd}
\usepackage{bbm}
\usepackage{mathtools}
\usepackage{setspace}
\usepackage{tikz,tikz-cd}
\usetikzlibrary{matrix,calc,positioning,arrows,decorations.pathreplacing,patterns,knots}

\newcommand{\la}{\langle}
\newcommand{\rg}{\rangle}

\newtheorem{theorem}{{Theorem}}[section]
\newtheorem*{theorem*}{Theorem}
\newtheorem{lemma}[theorem]{Lemma}
\newtheorem{proposition}[theorem]{Proposition}

\newtheorem{corollary}[theorem]{Corollary}
\newtheorem*{corollary*}{Corollary}

\theoremstyle{definition}

\newtheorem{remark}{Remark}

\usepackage{lmodern,url,enumerate,mathtools}
\usepackage[hmargin = 1in,vmargin=1in]{geometry}
\usepackage{graphicx}
\usepackage{subcaption}

\usepackage{hyperref}
    \hypersetup{colorlinks=true,citecolor=ForestGreen,linkcolor = blue,urlcolor =black,linkbordercolor={1 0 0}}

\newcommand{\ve}{\varepsilon}

\newcommand{\mr}[1]{{\rm #1}}
\newcommand{\mres}{\mathbin{\vrule height 1.6ex depth 0pt width
0.13ex\vrule height 0.13ex depth 0pt width 1.3ex}}

\newcommand{\cA}{\mathcal{A}}
\newcommand{\cD}{\mathcal{D}}

\newcommand{\cH}{\mathcal{H}}

\newcommand{\cL}{\mathcal{L}}
\newcommand{\cM}{\mathcal{M}}
\newcommand{\cP}{\mathcal{P}}

\newcommand{\bG}{\mathbb{G}}

\newcommand{\bR}{\mathbb{R}}
\newcommand{\bS}{\mathbb{S}}\newcommand{\bT}{\mathbb{T}}

\newcommand{\bZ}{\mathbb{Z}}

\newcommand{\nc}{\newcommand}
\nc{\sn}{\mr{sn}}
\nc{\cn}{\mr{cn}}
\nc{\dn}{\mr{dn}}

\nc{\on}{\operatorname}
\nc{\p}{\partial}
\nc{\ol}{\overline}
\nc{\ul}{\underline}
\nc{\pa}{\partial}

\nc{\pb}{\partial_b}
\nc{\pc}{\partial_c}
\nc{\pd}{\partial_d}
\nc{\pe}{\partial_e}
\nc{\pf}{\partial_f}
\nc{\pg}{\partial_g}
\nc{\ph}{\partial_h}
\nc{\pari}{\partial_i}
\nc{\pj}{\partial_j}
\nc{\pk}{\partial_k}
\nc{\pl}{\partial_l}
\nc{\pell}{\partial_\ell}
\nc{\parm}{\partial_m}
\nc{\pn}{\partial_n}
\nc{\po}{\partial_o}
\nc{\pp}{\partial_p}
\nc{\pq}{\partial_q}
\nc{\pr}{\partial_r}
\nc{\ps}{\partial_s}
\nc{\pt}{\partial_t}
\nc{\pu}{\partial_u}
\nc{\pv}{\partial_v}
\nc{\pw}{\partial_w}
\nc{\px}{\partial_x}
\nc{\py}{\partial_y}
\nc{\pz}{\partial_z}

\allowdisplaybreaks

\numberwithin{equation}{section}
\makeatletter
\@addtoreset{equation}{section}
\makeatother

\mathtoolsset{showonlyrefs}

\author[Benjy Firester]{Benjy Firester}
\address{{ \href{mailto:benjyfir@mit.edu}{benjyfir@mit.edu}} \hfill Department of Mathematics, MIT}
\author[Raphael Tsiamis]{Raphael Tsiamis}
\address{
{\href{mailto:r.tsiamis@columbia.edu}{r.tsiamis@columbia.edu}} \hfill Department of Mathematics, Columbia University}

\title[The anisotropic Michael-Simon inequality]{The anisotropic Michael-Simon inequality} %
\date{\today}

\begin{document}

\begin{abstract}
We prove a Michael-Simon inequality for varifolds of every dimension and codimension with respect to anisotropic energies. In particular, this resolves the problem for every convex even hypersurface anisotropy and, using work of Allard, yields the regularity of hypersurfaces with bounded anisotropic mean curvature in every dimension. By the results of De Philippis and Pigati, it also implies density bounds and compactness properties for rectifiable varifolds with uniformly bounded anisotropic first variation. The proof relies on a projection method for anisotropic stress measures as well as the recent resolution of the vanishing mass conjecture by Gennaioli and Rindler.
\end{abstract}

\maketitle

\vspace{-0.2in}

\section{Introduction}

We prove a Michael-Simon inequality for varifolds in all dimensions and codimensions inside Euclidean space with respect to anisotropic energies.
Varifolds are generalizations of submanifolds and form the foundation of geometric measure theory.
For varifolds whose first variation with respect to the area functional is bounded, Allard proved deep and highly influential compactness, rectifiability, and regularity theorems~\cites{allard-first-variation , allard-regularity }.
Among other properties, this property relies on the monotonicity formula, which fails for energies defined by general anisotropic integrands; this failure makes density, compactness, and regularity properties central questions that require novel techniques~\cite{allard-monotonicity}.
In the absence of a monotonicity formula for anisotropic integrands, a Michael-Simon inequality~\cite{michael-simon} is the foremost tool for establishing such properties.
Despite breakthroughs on the isotropic Michael-Simon inequality~\cites{brendle-sharp-isoperimetric , brendle-eichmair }, the anisotropic problem has remained open.
In their beautiful recent work, De Philippis and Pigati~\cite{anisotropic-michael-simon} utilized ideas of Almgren to prove a Michael-Simon inequality for surfaces in $\bR^3$, with respect to anisotropies sufficiently close to the area functional.
They also showed, for integrands satisfying the atomic condition, that the Michael-Simon inequality is equivalent to compactness of rectifiable varifolds.

Our result applies to anisotropic integrands satisfying a very general condition, which includes a $C^1$ neighborhood of every convex even hypersurface anisotropy, as well as of every $\ell^p$ norm for $2$-varifolds in arbitrary codimension. 
For simplicity, we first state the result near the area functional.  

\begin{theorem}\label{thm:anisotropic-MS}
For every $2 \leq m < N$, there exist constants $\ve_{N,m}>0$ and $C_N<\infty$ with the following property.
For every $\Psi \in C^1(\bG(N,m) , (0,\infty) )$ with $\sup_{T \in \bG(N,m)} \|B_{\Psi}(T) - T \| < \ve_{N,m}$, every rectifiable $m$-varifold $V = v(M,\theta)$ in $\bR^N$ with finite total mass and finite first variation with respect to the anisotropy $\Psi$ satisfies
\[
\|V\| (\bR^N) \leq C_N \, \cH^m(M)^{\frac{1}{m}} \, |\delta_{\Psi} V|(\bR^N).
\]
In particular, if $\Theta^m(\|V\|,x) \geq \theta_0 > 0$ holds $\|V\|$-a.e.~$x$, then 
\[
\|V\|(\bR^N)^{\frac{m-1}{m}} \leq C_N \theta_0^{- \frac{1}{m}} |\delta_{\Psi} V|( \bR^N).
\]
\end{theorem}

This theorem is a particular instance of a more general framework, valid for every $C^1$ anisotropy satisfying a spectral condition that is open in the $C^1$ topology; we formulate this general result as Theorem~\ref{thm:most-general-michael-simon}. 
The general spectral condition established in our comprehensive Theorem~\ref{thm:most-general-michael-simon} implies the Michael-Simon inequality for every anisotropy whose stress tensor $B_{\Psi}$ has positive average, under conjugation by elements on $\bG(N,m)$, upon applying a concave one-homogeneous function $\phi$.
In Section~\ref{section:smallest-eigenvalue}, we introduce such functions, which we call \textbf{concave barriers} after the terminology of~\cite{vanishing-mass-conjecture}*{\S 2}.
This general framework leads to several new applications and Michael-Simon inequalities for natural anisotropic energies.
Notably, we prove the following.
\begin{theorem}\label{thm:even-convex}
Let $\Psi_0 \in C^1( \bG(m+1,m))$ be a positive, convex, even $C^1$ hypersurface anisotropy.
There exists an $\ve_0 = \ve_0(\Psi_0)$ such that for every $\Psi \in C^1(\bG(m+1,m))$ with $\| \Psi - \Psi_0\|_{C^1} \leq \ve_0$, not necessarily convex, and every rectifiable $m$-varifold $V = v(M,\theta)$ in $\bR^{m+1}$ with finite total mass and finite first variation with respect to $\Psi$, it holds that
\[
\|V\|(\bR^{m+1}) \leq C(m,\Psi_0) \, \cH^m(M)^{\frac{1}{m}} |\delta_{\Psi} V| (\bR^{m+1}).
\]
\end{theorem}
Combined with Corollaries~\ref{cor:allard} and~\ref{cor:density}, this result resolves the question of a Michael-Simon inequality, therefore also density estimates, compactness, rectifiability, and regularity, for every convex even hypersurface anisotropy.
Interestingly, we do not require $\Psi$ itself to be even or convex.

We also prove the following result for anisotropic integrands based on the $\ell^p$ norms.
\begin{theorem}\label{thm:most-important-application}
The Michael-Simon inequality of Theorem~\ref{thm:anisotropic-MS}, and hence the consequences of Corollary~\ref{cor:density}, for $2$-varifolds in $\bR^N$ with respect to the $\ell^p$ anisotropy, for every $N \geq 3$ and $p \in (1,\infty)$.
\end{theorem}
De Philippis-Pigati~\cite{anisotropic-michael-simon }*{Theorem 1.4} proved Theorems~\ref{thm:even-convex} and~\ref{thm:most-important-application} when $N=3$, for the $\ell^p$ anisotropy and for anisotropic integrands that are reflection-symmetric in every coordinate.

Theorem~\ref{thm:even-convex} immediately implies the density lower bound assumption in Allard's anisotropic regularity Lemma~\cite{allard-regularity}*{1, Basic regularity Lemma, Assumption (1)}.
Using the regularity theorem~\cite{allard-regularity} for codimension-$1$ varifolds in the anisotropic setting of~\cite{allard-regularity}, we obtain the following.
\begin{corollary}\label{cor:allard}
Let $\Psi_0 \in C^1( \bG(m+1,m))$ be a positive, convex, even $C^1$ hypersurface anisotropy.
There exists an $\ve_0 = \ve_0(m,\Psi_0)$ such that for every $\Psi \in C^1(\bG(m+1,m))$ with $\| \Psi - \Psi_0\|_{C^1} \leq \ve_0$, the following property holds.
Let $r>0$, let $x_0\in\bR^{m+1}$, let $P_0\in\bG(m+1,m)$, and let $V$ be an integral $m$-varifold in
$B_{2r}(x_0)$ with $x_0 \in \textup{spt} \, \|V\|$.
Suppose that
\begin{align*}
|\delta_{\Psi} V| \leq \Lambda \|V\| \quad \text{for } \; 0 \leq r \Lambda < \ve_0&, \qquad \frac{\|V\|(B_{2r}(x_0))}{\omega_m (2r)^m} \in \Bigl[ \frac{1}{2}, \frac{3}{2} \Bigr], \\
\int_{B_{2r}(x_0)} | \pi_{P_0^{\perp}}(x - x_0)|^2 \, d \|V\|(x) &\leq \ve_0 r^{m+2}.
\end{align*}
Then, in the ball $B_r(x_0)$, the varifold $V$ agrees with the multiplicity-one varifold associated to the graph of a $C^{1,\alpha}$ function $u$ over the affine plane $x_0 + P_0$, for every $\alpha \in (0,1)$.
Moreover,
\[
r^{-1} \|u \|_{C^0} + \|Du \|_{C^0} + r^{\alpha} [ Du]_{C^{0,\alpha}} \leq \eta_{\alpha}(t), \qquad \eta_{\alpha}(t) \to 0 \quad \text{as } \; t \downarrow 0,
\]
when the corresponding mass excess, first variation, and tilt-excess are bounded by $t$.
\end{corollary}
De Rosa and Tione~\cite{derosa-tione-regularity} established this regularity theorem in every dimension and codimension for varifolds associated to Lipschitz graphs.
When $m=2$, De Philippis and Pigati~\cite{anisotropic-michael-simon} obtained this result for anisotropies close to the area using their Michael-Simon inequality.
Finally, Kolasi\'nski-Santili~\cite{MarioSlawomir} proved a $C^{1,\alpha}$ regularity result for hypersurfaces provided the anisotropic integrand is $C^3$ near a.e.~density $1$ points. 

The consequences of Theorem~\ref{thm:anisotropic-MS} in arbitrary codimension transfer to every anisotropy satisfying the conditions of Theorem~\ref{thm:most-general-michael-simon}; for simplicity, we state them for anisotropies close to the area functional.
A first well-known consequence is the functional Michael-Simon inequality.

\begin{corollary}\label{cor:functional-MS}
Under the assumptions of Theorem~\ref{thm:anisotropic-MS}, suppose in addition that $\Theta^m(\|V\|,x)\geq\theta_0>0$ for $\|V\|$-a.e.~$x$.
Then, for every compactly supported Lipschitz function $f$ that is defined in an open neighborhood of $\textup{spt} \, \|V\|$, it holds that
\[
\Bigl( \int |f|^{\frac{m}{m-1}} \, d \|V\| \Bigr)^{\frac{m-1}{m}} \leq C_N \, \theta_0^{- \frac{1}{m}} \Bigl[ \int |\nabla^V f| \, d \|V\| + \int |f| \, d|\delta_{\Psi} V| \Bigr].
\]
\end{corollary}

For anisotropic integrands satisfying the atomic condition, De Philippis and Pigati~\cite{anisotropic-michael-simon}*{\S 2} derived several equivalent formulations of the Michael-Simon inequality; in particular, they proved that it implies density bounds and compactness of rectifiable varifolds.
Combining their results with Theorem~\ref{thm:anisotropic-MS}
gives the following consequences in every dimension and codimension.
\begin{corollary}\label{cor:density}
Let $\Psi \in C^1(\bG(N,m) , (0,\infty) )$ be a convex anisotropy with $\sup_{T \in \bG(N,m)} \|B_{\Psi}(T) - T \| < \ve_{N,m}$.
    Then, the following properties are satisfied.
\begin{enumerate}[(i)]
    \item There exists a constant $c(N,m,\Psi)>0$ such that every rectifiable $m$-varifold $V$ in $\bR^N$ with $\Theta^m( \|V\|,x) \geq \theta_0$ with anisotropic first variation locally in $L^p$ for $p>m$ satisfies $\Theta^m_* (\|V\|,x) \geq c \theta_0$ for every $x \in \textup{spt} \, \|V\|$.
    \item Let $V_i \rightharpoonup V$ be a converging sequence of rectifiable $m$-varifolds satisfying
    \[
    \Theta^m ( \|V_i \|, x) \geq \theta_0 > 0 \qquad \text{for } \; \|V_i\|\textup{-a.e.~} \; x
    \]
    and having locally uniformly bounded anisotropic mean curvature in $L^p$ for $p>m$.
    Then, $\textup{spt} \, \|V_i \| \to \textup{spt} \, \|V\|$ locally in the Hausdorff distance.
    In particular, the supports converge in the Hausdorff distance on $\bT^N$, or in $\bR^N$ if they are contained in a common compact set.
\end{enumerate}
Assume in addition that $\Psi$ satisfies the atomic condition.
Then the following properties hold.
\begin{enumerate}[(i)]
    \item[(iii)] Let $V_i$ be rectifiable $m$-varifolds in $\bR^N$, or equivalently in the flat torus $\bT^N = \mathbb{R}^N / \mathbb{Z}^N$, with locally uniformly bounded anisotropic first variation and $\Theta^m ( \|V_i \|,x) \geq \theta_0>0$ for $\|V_i\|$-a.e.~$x$.
    If $V_i \rightharpoonup V$ as varifolds, then $V$ is rectifiable and $\Theta^m ( \|V\|,x) \geq \theta_0$ for $\|V\|$-a.e.~$x$.
    Equivalently, there is no sequence of rectifiable $m$-varifolds in $\bR^N$ with uniformly bounded anisotropic first variation that exhibit diffuse concentration.
    \item[(iv)] Allard's strong constancy lemma holds for anisotropic varifolds, cf.~\cite{allard-regularity}*{p.3} and~\cites{derosa-ghiraldin , dephilippis-rindler-a-free }.
    Namely, let $V_i \rightharpoonup V$ be a converging sequence of rectifiable $m$-varifolds in $\bR^N$ or $\bT^N$.
    For every $i$, let $\cD_i \subset (0,\infty)$ satisfy $\inf \cD_i > 0$ and consider Borel sets $E_i \subseteq \bR^N$ such that $\Theta^m ( \|V_i \|, x) \in \cD_i$ for $\|V_i\|$-a.e.~$x \not\in E_i$ and
    \[
    \|V_i \|(E_i \cap K) \to 0 \qquad \text{and} \qquad \textstyle \sup_i |\delta_{\Psi} V_i|(K) < \infty
    \]
    for every compact set $K$.
    For $\cD \subset (0,\infty)$, write $S(\cD)$ for the finite-sum set of $\cD \subseteq \bR$, i.e., 
    \[
    S(\cD) := \{ d_1 + \cdots + d_q : q \geq 1, \; d_1, \dots, d_q \in \cD \}.
    \]
    Then, after passing to a subsequence, for $\|V\|$-a.e.~$x$ there are values $v_i(x) \in S(\cD_i)$ such that $\Theta^m (\|V\|,x) = \lim_{i \to \infty} v_i(x)$.
    \item[(v)] The class of integer rectifiable $m$-varifolds in $\bR^N$ is closed under varifold convergence with locally uniformly bounded anisotropic first variation.
\end{enumerate}
All the above conclusions are local.
In particular, they hold for any non-autonomous anisotropy $\Psi(x,T)$ on $\bR^N$ and any anisotropic integrand on the Grassmannian bundle of a smooth $N$-manifold whenever the corresponding Theorem~\ref{thm:anisotropic-MS} is available on compact subsets.
\end{corollary}

\subsection{Strategy of the proof}

Let us explain the main idea of the proof.
In Section~\ref{section:smallest-eigenvalue}, we introduce a projection method for anisotropic stress measures.
Importantly, the projected measure satisfies support, mass, and divergence bounds in terms of the corresponding quantities for the varifold $V$, proved in Lemma~\ref{lemma:A-P-projection-properties}.
This technique is motivated by Brendle's ingenious formulation of an optimal transportation scheme for the sharp isoperimetric inequality in higher codimension, developed further in work with Eichmair~\cites{brendle-sharp-isoperimetric , brendle-eichmair }. 
In Brendle's argument, optimal transportation is combined with the normal bundle to lift the problem onto the ambient space, equipped with a rotationally symmetric density.
In our setting, the analogous role is played by the normalized Haar measure on the Grassmannian $\bG(N,m)$.

Using this projection method, we reduce the anisotropic Michael-Simon inequality for $m$-varifolds in arbitrary codimension to a quantitative non-concentration problem for matrix-valued measures with controlled divergence on $\bR^m$.
The construction of projected stress measures extends, to every dimension and codimension, the approach of De Philippis and Pigati~\cite{anisotropic-michael-simon}, which employed the Cahn-Hoffman vector of $2$-varifolds in $\mathbb{R}^3$.
In a precursor paper~\cite{anisotropic-surfaces}, we carry out this approach to prove a Michael-Simon inequality for $2$-varifolds in every codimension, using the multilinear Kakeya inequality of~\cite{anisotropic-michael-simon}.

Very recently, the remarkable work of Gennaioli and Rindler~\cite{vanishing-mass-conjecture} introduced several deep ideas to prove Bouchitt\'e's vanishing mass conjecture~\cite{shape-optimization}.
The vanishing mass theorem is intimately connected to the Michael-Simon inequality and the compactness of rectifiable varifolds and vector-valued measures.
Among many other consequences, this result, for the operator $\cA = \textup{curl}$, implies Alberti's rank one theorem~\cites{albert-rank-one , alberti-csornyei-preiss , delellis-rank-one }.
We note that the multilinear Kakeya inequality of De Philippis and Pigati also implies Alberti's rank one theorem.
In Section~\ref{section:measure-theory}, we derive certain important consequences of the results in~\cite{vanishing-mass-conjecture}, notably a quantitative non-concentration result for vector-valued measures; see Corollary~\ref{cor:concave-barrier}.
The theorem follows from combining this property with the projection estimate of Lemma~\ref{lemma:A-P-projection-properties} and an averaging inequality established in Lemma~\ref{lem:averagedCoercivityGeneral}.

\smallskip \noindent \textbf{Acknowledgments.}
We are very grateful to Guido de Philippis and Alessandro Pigati for valuable comments and insights on a previous version of this manuscript, providing an elegant and simpler argument which strengthens Theorems~\ref{thm:even-convex} and~\ref{thm:most-important-application} and their consequences to a $C^1$ neighborhood of every convex even hypersurface anisotropy.

We are very thankful to Antonio de Rosa for many helpful discussions on anisotropic problems and for insightful comments on a preliminary version of this manuscript.
We also thank Simon Brendle and Toby Colding for inspiring conversations.
BF recognizes support from a Simons Dissertation Fellowship, a MathWorks Fellowship, and the Citadel GQS PhD Fellowship.
RT is supported by a Simons Dissertation Fellowship, an A.G.~Leventis Foundation Scholarship, and an Onassis Foundation Scholarship.

\section{Preliminaries}

Let $\Psi: \bG(N,m) \to (0,+\infty)$ be $C^1$.
We denote by $A^t$ the transpose of an endomorphism $A$ and identify $\bG(N,m)$ with the set of rank-$m$ orthogonal projections in $\bR^N$, namely 
\[
\bG(N,m) = \{ T \in \bR^{N \times N} : T = T^t, \; T^2 = T, \; \on{tr} T = m \}.
\]
Given an $m$-plane $T \in \bG(N,m)$, we define a tensor field $B_{\Psi}(T)$ by
\begin{equation}\label{eqn:BPsi(T)}
B_{\Psi}(T) = \Psi(T) T + T^{\perp} \, d \Psi(T) \, T, \qquad \text{where } \; d \Psi(T) := D \Psi(T) + D \Psi(T)^t.
\end{equation}
By this definition, $B_{\Psi}(T)$ satisfies $B_{\Psi}(T) T = B_{\Psi}(T)$ and
\begin{equation}\label{eqn:operator-norm}
    \sup_{T \in \bG(N,m)} \| B_{\Psi}(T) - T \| \leq \| \Psi - 1 \|_{C^0} + 2 \, \|D \Psi\|_{C^0} \leq 2 \, \| \Psi - 1 \|_{C^1}
\end{equation}
in operator norm.
Indeed, $\| T^{\perp} \, d \Psi(T) \, T \| \leq 2 \, \|D\Psi(T)\|$ because $T,T^{\perp}$ are orthogonal projections.
In particular, $\sup_{T \in \bG(N,m)} \|B_{\Psi}(T) \| \leq C_N$ whenever $\sup_{T \in \bG(N,m)} \|B_{\Psi}(T) - T \| \leq c_N$.

Let $V = v(M,\theta)$ be a rectifiable $m$-varifold in an open set $\Omega \subset \bR^N$, namely a Radon measure on $\Omega \times \bG(N,m)$.
We write $\|V\|$ for the weight measure, which satisfies $\|V\|(A) := V(A \times \bG(N,m))$ for every Borel set $A \subseteq \Omega$.
If $\Theta^m(\|V\|,x_0) \geq \theta_0$ a.e., then we can choose the rectifiable set $M$ to satisfy $\theta \geq \theta_0$ $\cH^m$-a.e., and hence $\cH^m(M) \leq \theta_0^{-1} \|V\|(\Omega)$.

The \textbf{anisotropic first variation} of $V$ with respect to the flow generated by a Lipschitz vector field $g \in \textup{Lip}_c(\Omega;\bR^N)$ is computed in~\cite{derosa-ghiraldin}*{Lemma A.2} as
\begin{equation}\label{eqn:first-variation}
[\delta_{\Psi} V] (g) = \int \la B_{\Psi}(T), D g(x) \rg \, dV(x,T) =
\int_{M} \la B_{\Psi} (T_x M), Dg(x) \rg \, \theta \, d \cH^m(x).
\end{equation}
Finally, we say that the anisotropy $\Psi$ satisfies the \textbf{atomic condition (AC)} if, for every probability measure $\lambda$ on $\bG(N,m) $ and $A_{\Psi}(\lambda) := \int_{\bG(N,m)} B_{\Psi}(T) \,  \lambda(T)$, we have $\textup{rank} \, A_{\Psi}(\lambda) \geq m$ with equality if and only if $\lambda = \delta_{T_0}$ for some $T_0 \in \bG(N,m)$.
We refer the reader to the foundational works~\cites{derosa-ghiraldin , derosa-kolasinski , derosa-tione-regularity } for key consequences and equivalent conditions of (AC).

For varifolds in codimension one, every tangent hyperplane $T \in \bG(N,N-1)$ can identified with its unit normal $\pm \nu \in \bS^{N-1}$, so that $T = \nu^{\perp}$.
Accordingly, an anisotropic integrand $\Psi$ on $\bG(N,N-1)$ can be equivalently identified via $F(\nu) := \Psi(\nu^{\perp})$ with an even, positively one-homogeneous function $F: \bR^N \to [0,\infty)$ with stress tensor $B_F(\nu) = F(\nu) I - \nu \otimes dF(\nu)$.
Since $F$ is even, $B_F(-\nu) = B_F(\nu)$ is well-defined on the unoriented hyperplane $T = \nu^{\perp}$.
This formulation will be relevant for Theorem~\ref{thm:most-important-application}.

When $F : \bR^N \to [0,\infty)$ is convex, we can define the dual norm $F^*$ by 
\[ 
F^*(\ell) := \sup_{F(y) = 1} |\ell(y)| . 
\] 
For every $\nu \neq 0$, convexity and one-homogeneity give $dF(\nu)[y] \leq F(y)$ for every $y \in \bR^N$, and hence $F^*(dF(\nu)) \leq 1$.
On the other hand, Euler's identity gives $dF(\nu)[\nu] = F(\nu)$.
Since $F( \frac{\nu}{F(\nu)}) = 1$, evaluating the dual norm at $\frac{\nu}{F(\nu)}$ gives the reverse inequality.
Therefore,
\begin{equation}\label{eqn:F-star-dF-nu}
    F^* (dF(\nu)) = 1.
\end{equation}
Equivalently, $dF(\nu)$ is the supporting functional to the unit ball $\{ x : F(x) = 1 \}$ at the point $\frac{x}{F(x)}$. 

\subsection{Concave barriers and projected stresses}\label{section:smallest-eigenvalue}
Let $P$ be an $m$-dimensional Euclidean space.
A continuous function $\phi : \textup{End}(A) \to \bR$ is called a \textbf{concave barrier} if it is concave, positively one-homogeneous, and satisfies
\begin{equation}\label{eqn:concave-barrier}
    \phi(tA) = t \, \phi(A) \qquad \text{for all } \; t \geq 0, \qquad \phi(A) \leq 0 \qquad \text{whenever } \; \det A=0.
\end{equation}
Given a concave barrier $\phi$, the function $q := - \phi$ is positively one-homogeneous and convex, hence subadditive; moreover, bounding $|q|$ on the unit sphere of $\textup{End}(P)$ shows that
\begin{equation}\label{eqn:linear-bound}
    |\phi(A) - \phi(B)| \leq C_{\phi} |A-B|, \qquad \text{for a constant } \; C_{\phi}.
\end{equation}
The second property is significant in relation to the vanishing mass theorem~\cite{vanishing-mass-conjecture}.
Given a first-order differential operator $\mathbb{A}$, the \textbf{wave cone} was introduced in Tartar's theory of compensated compactness as the collection of all vectors $v \in V$ such that $\mathbb{A} v$ is not elliptic; see~\cite{rindler-book}*{\S 8.2}.
For the row-wise divergence of $m \times m$ matrices, the symbol is $\mathbb{A}(\xi) = A \xi$, and hence the wave cone is
\begin{equation}\label{eqn:lambda-div}
    \Lambda_{\textup{Div}} = \{ A \in \bR^{m \times m} : \det A = 0 \}.
\end{equation}
Consequently, the property~\ref{eqn:concave-barrier} implies that the function $q := - \phi$ is convex, positively one-homogeneous, and non-negative on $\Lambda_{\textup{Div}}$.
It is therefore a barrier function in the sense of~\cite{vanishing-mass-conjecture}*{\S 2}.

Consider now a finite $\textup{End}(P)$-valued measure $\mathbf{A} \in \cM( \bR^m; \textup{End}(P))$ and $\sigma$ any positive measure dominating the entries of $\mathbf{A}$.
For $a := \frac{d \mathbf{A}}{d \sigma}$ the Radon-Nikodym derivative, we define the functional 
\begin{equation}\label{eqn:Lambda-m-definition}
\Phi(\mathbf{A}) := \int \phi(a)\,d\sigma , \qquad a := \frac{d \mathbf{A}}{d \sigma}
\end{equation}
which is independent of the choice of $\sigma$, due to the $1$-homogeneity of $\phi$.

In particular, the smallest eigenvalue function $\lambda_m$ defines a concave barrier.
\begin{lemma}\label{lemma:lambda-m-properties}
For $A \in \textup{End}(P)$, we will consider the smallest eigenvalue function
\[
    \lambda_m(A) := \lambda_{\min} ( \textup{sym} \, A), \qquad \text{where } \; \textup{sym} \, A := \tfrac{1}{2} (A + A^t).
\]
Equivalently, $\lambda_m(A) = \min_{ |v| = 1 } \la A v , v \rg$.
The function $\lambda_m$ is superadditive on matrices and defines a concave barrier with $|\lambda_m(A) - \lambda_m(B)| \leq \|A - B\|_{\textup{op}}$ and $\lambda_m (R^t AR) = \lambda_m(A)$ for every $R \in O(P)$.
\end{lemma}
\begin{proof}
    We only prove the second part of~\eqref{eqn:concave-barrier}, as all other properties are immediate.
    For a matrix $A$, if $\lambda_{\min} ( \textup{sym} \, A) > 0$, then $\la Av, v \rg = \la \textup{sym} \, A \, v, v \rg > 0$ for every $v \neq 0$.
    Thus, $A$ is injective, and hence invertible, so $\det A \neq 0$.
    Consequently, $\det A = 0$ implies $\lambda_m(A) \leq 0$.
\end{proof}
We fix a collection of isometries $E_P : \bR^m \to P$ for each $m$-plane $P \in \bG(N,m)$ such that the association $P \mapsto E_P$ is Borel and write $\pi_P := E_P^t : \bR^N \to \bR^m$.
We introduce a projected stress measure $\mathbf{A}_P \in \cM( \bR^m ; \bR^{m \times m})$ by
\begin{equation}\label{eqn:projected-stress-measure}
\begin{split}
    \mathbf{A}_P &:= (\pi_P)_{\#} \bigl[ E_P^t B_{\Psi}(T) E_P \, d V(x,T) \bigr], \qquad \textup{so that} \\
    \int_{\bR^m} \la \varphi(y), d \mathbf{A}_P(y) \rg &\;= \int_{\bR^N \times \bG(N,m)} \la \varphi(\pi_P x), E^t_P B_{\Psi}(T) E_P \rg \, d V(x,T)
\end{split}
\end{equation}
for $\varphi \in C_c(\bR^m ; \bR^{m \times m})$.
Replacing $E_P$ by $E_P R$ for some $R \in O(m)$ conjugates the matrix measure by an orthogonal map, so the corresponding measure is obtained from $\mathbf{A}_P$ by simultaneously rotating the domain and conjugating its values.
In particular, all quantities below involving total variation and divergence are independent of the choice of $E_P$.

For a rectifiable varifold $V$ satisfying $\Theta^m ( \|V\|, x) \geq \theta_0$ for $\|V\|$-a.e.~$x$, we can write $V = v(M, \theta)$ where $\theta \geq \theta_0$ $\cH^m$-a.e.~on the rectifiable set $M$.
Then, $\mathbf{A}_P$ is concentrated on $\pi_P(M)$, and
\[
\cL^m ( \pi_P(M)) \leq \int_M J_m (\pi_P|_{T_x M}) \, d \cH^m(x) \leq \cH^m(M)
\]
by the area formula.
Since $\pi_P$ is continuous, while $M$ is rectifiable and may be chosen Borel, the set $\pi_P(M)$ is Lebesgue measurable.
Thus, we can choose a Borel set $G_P \supset \pi_P(M)$ such that $\mathbf{A}_P$ is concentrated on $G_P$ and $\cL^m (G_P) = \cL^m(\pi_P(M)) \leq \cH^m(M)$.

For a matrix-valued measure $\mathbf{A}$ on $\bR^m$, we define its row-wise distributional divergence by
\begin{equation}\label{eqn:divergence-definition}
\la \textup{Div} \, \mathbf{A} , X \rg := - \int_{\bR^m} \la \mathbf{A}, DX \rg, \qquad \text{for } \; X \in C_c^1(\bR^m; \bR^m).
\end{equation}
When the anisotropic first variation $\delta_{\Psi} V$ defines a finite measure on $\bR^N$, our definitions imply that the projected stress measure~\eqref{eqn:projected-stress-measure} satisfies $\textup{Div} \, \mathbf{A}_P = - (\pi_P){\#} (E^t_P \nu)$.
\begin{lemma}\label{lemma:A-P-projection-properties}
Suppose that $\|V\|(\bR^N)<\infty$ and that $V$ has finite anisotropic first variation with respect to an anisotropy $\Psi$ satisfying $\sup_{T \in \bG(N,m)} \|B_{\Psi}(T) - T \| \leq C_N$.
Then, for every $P \in \bG(N,m)$, the measure $\textup{Div} \, \mathbf{A}_P$ has finite total variation and
    \[
    |\textup{Div} \, \mathbf{A}_P| (\bR^m) \leq |\delta_{\Psi} V|(\bR^N), \qquad |\mathbf{A}_P|(\bR^m) \leq C'_N \, \|V\|(\bR^N).
    \]
\end{lemma}
\begin{proof}
Let us write $C'_N := \sup_{T \in \bG(N,m)} \|B_{\Psi}(T) \|_{\textup{op}} < \infty$.
For every $X \in C_c^1(\bR^m ; \bR^m)$, the lift $\tilde{X}(x) := E_P X( \pi_P x)$ is bounded and constant along the directions in $P^{\perp}$.
We choose cutoff functions $\eta_R \in C_c^1(\bR^N)$ with $0 \leq \eta_R \leq 1$ and $\eta_R \equiv 1$ on $B_R$, $\eta_R \equiv 0$ outside $B_{2R}$, and $|\nabla \eta_R| \leq CR^{-1}$.
Let
\[
X_R := \eta_R \tilde{X} \in C_c^1(\bR^N; \bR^N), \qquad D \tilde{X}(x) = E_P \, DX( \pi_P x) \, E^t_P.
\]
Thus, the first variation formula for the vector field $X_R$ gives
\begin{equation}\label{eqn:delta-Psi-V-XR}
    \begin{split}
        [\delta_{\Psi} V](X_R) &= \int \eta_R(x) \la B_{\Psi}(T), E_P \, DX (\pi_P x) \, E^t_P \rg \, dV(x,T) \\
        & \quad + \int \la B_{\Psi}(T), \tilde{X}(x) \otimes D \eta_R(x) \rg \, dV(x,T).
    \end{split}
\end{equation}
The second term in the above expression satisfies $|\int \la B_{\Psi}(T), \tilde{X} \otimes D \eta_R \rg \, dV| \leq CC'_N R^{-1} \|X\|_{L^{\infty}} \|V\|(\bR^N)$, and hence tends to zero as $R \to \infty$.
For the first term above, the definition~\eqref{eqn:projected-stress-measure} of $\mathbf{A}_P$ and an application of the dominated convergence theorem give
\[
\lim_{R \to \infty} \int \eta_R(x) \la B_{\Psi}(T), E_P \, DX(\pi_P x) \, E^t_P \rg \, dV(x,T) = \int_{\bR^m} \la DX(y), d \mathbf{A}_P(y) \rg.
\]
On the other hand, 
\[
|\delta_{\Psi} V (X_R)| \leq \| X_R \|_{L^{\infty}} |\delta_{\Psi} V|(\bR^N) \leq \|X \|_{L^{\infty}} |\delta_{\Psi} V| (\bR^N).
\]
Passing to the limit as $R \to \infty$ in~\eqref{eqn:delta-Psi-V-XR} therefore yields
\[
\Bigl| \int_{\bR^m} \la D X, d \mathbf{A}_P \rg \Bigr| \leq \|X \|_{L^{\infty}} |\delta_{\Psi} V|(\bR^N).
\]
This means that the distribution $X \mapsto \la \mr{Div} \mathbf{A}_P , X \rg := - \int_{\bR^m} \la D X, d \mathbf{A}_P \rg$ is of order zero, and hence represented by a finite $\bR^m$-valued Radon measure.
Taking the supremum over $X \in C_c^1(\bR^m; \bR^m)$ with $\|X\|_{L^{\infty}}$ proves the claimed bound $|\textup{Div} \, \mathbf{A}_P|(\bR^m) \leq |\delta_{\Psi} V|(\bR^N)$.

Finally, since projection by $E_P$ does not increase the operator norm, we have $\|E^t_P B_{\Psi}(T) E_P \| \leq C'_N$.
Consequently, by the definition of the pushforward measure,
\[
|\mathbf{A}_P| (\bR^m) \leq C_m \int_{\bR^N \times \bG(N,m)} \| E^t_P B_{\Psi}(T) E_P \| \, d V(x,T) \leq C_m C'_N \|V\|(\bR^N)
\]
for a dimensional constant $C_m$.
This completes the proof.
\end{proof}

Concave barriers have the crucial property that contributions from different
sheets of $V$ lying over the same point of the projection giving rise to $\mathbf{A}_P$ increase the corresponding spectral functional.
\begin{lemma}\label{lemma:concavity-AP}
Let $\phi$ be a concave barrier with associated function $\Phi$ on measures, defined in~\eqref{eqn:Lambda-m-definition}.
    For every $m$-plane $P \in \bG(N,m)$, it holds that
    \[
    \Phi ( \mathbf{A}_P) \geq \int_{\bR^N \times \bG(N,m)} \phi( E^t_P B_{\Psi}(T) E_P) \, d V(x,T).
    \]
\end{lemma}
\begin{proof}
Consider the projection $\tilde{\pi}_P : \bR^N \times \bG(N,m) \to \bR^m$ given by $\tilde{\pi}_P(x,T) := \pi_P x$, and let $\mu_P := (\pi_P)_{\#} \|V\| = (\tilde{\pi}_P)_{\#} V$.
The maps $\tilde{\pi}_P$ and $\pi_P$ are continuous.
By the disintegration theorem for rectifiable varifolds~\cite{simon-gmt}*{Ch.~8}, we can find a $\mu_P$-measurable family of probability measures $\{ \sigma_{P,y}\}_{y \in \bR^m}$, with $\sigma_{P,y}$ concentrated on $\tilde{\pi}^{-1}_P(y)$ for $\mu_P$-a.e.~$y$, such that
\begin{equation}\label{eqn:disintegration-of-measures}
    \int f(x,T) \, dV(x,T) = \int_{\bR^m} \Bigl( \int f(x,T) \, d \sigma_{P,y}(x,T) \Bigr) \, d\mu_P(y)
\end{equation}
for every $V$-integrable Borel function $f$.
It follows from the definition~\eqref{eqn:projected-stress-measure} of $\mathbf{A}_P$ and Lemma~\ref{lemma:A-P-projection-properties} that $\mathbf{A}_P \ll \mu_P$, and $\mu_P$-a.e.~$y$ satisfies
\begin{equation}\label{eqn:density-AP-muP}
    \frac{d \mathbf{A}_P}{d \mu_P}(y) = \int_{\tilde{\pi}^{-1}_P(y)} E^t_P \, B_{\Psi}(T) \, E_P \, d \sigma_{P,y}(x,T).
\end{equation}
Indeed, this property is obtained by testing the right-hand side against a compactly supported matrix-valued function and using~\eqref{eqn:disintegration-of-measures}.
Since the definition of $\Phi$ is independent of the choice of dominating measure, and $\mathbf{A}_P \ll \mu_P$, we can use $\mu_P$ in the definition~\eqref{eqn:Lambda-m-definition}.
Because $\phi$ is concave and $\sigma_{P,y}$ is a probability measure, Jensen's inequality gives
\begin{align*}
    \Phi ( \mathbf{A}_P) &= \int_{\bR^m} \phi \Bigl( \int E^t_P B_{\Psi}(T) E_P \, d \sigma_{P,y}(x,T) \Bigr) \, d \mu_P(y) \\
    &\geq \int_{\bR^m} \int \phi(E^t_P B_{\Psi}(T) E_P) \, d \sigma_{P,y}(x,T) \, d\mu_P(y) \\
    &= \int_{\bR^N \times \bG(N,m)} \phi(E^t_P B_{\Psi}(T) E_P) \, d V(x,T).
\end{align*}
This completes the proof of our assertion.
\end{proof}

We will study the average of the functional $\Phi ( \mathbf{A}_P)$ over projections onto $m$-planes in $\bG(N,m)$.
For a concave barrier $\phi$, a Borel set $\Gamma\subset\bG(N,m)$, and a probability measure $\rho$ on $\bG(N,m)$, define
\begin{equation}\label{eqn:kappa-Psi-Gamma-phi-rho}
\kappa(\Psi, \Gamma, \phi, \rho) := \inf_{T\in\Gamma} \int_{\bG(N,m)} \phi (E_P^tB_\Psi(T)E_P) \, d \rho(P)
\end{equation}
using our fixed choice of isometries $E_P : \bR^m \to P$ so that the map $P \mapsto E_P$ is Borel.

\begin{lemma}\label{lem:averagedCoercivityGeneral}
Let $V=v(M,\theta)$ be a finite rectifiable $m$-varifold such that
$T_xM\in\Gamma$ for $\cH^m$-a.e.~$x\in M$.
Then, there exists
$P\in\bG(N,m)$ such that $\Phi(\mathbf A_P) \geq \kappa(\Psi, \Gamma, \phi, \rho) \|V\| (\bR^N)$.
\end{lemma}
\begin{proof}
Because the map $P \mapsto E_P$ is Borel, and both $B_{\Psi}, \phi$ are continuous, we can define the map
\[
F_V(P) := \int_{\bR^N\times\bG(N,m)} \phi (E_P^tB_\Psi(T)E_P) \,dV(x,T)
\]
for $P \in \bG(N,m)$, which is Borel measurable.
Because $\Psi$ is $C^1$ on $\bG(N,m)$ and $\phi$ satisfies the linear bound~\eqref{eqn:linear-bound}, the integrand is uniformly bounded.
Let us denote $\kappa = \kappa(\Psi, \Gamma, \phi, \rho)$.
Applying Fubini's theorem and using the fact that $T\in\Gamma$ for $V$-a.e.~$(x,T)$, we obtain 
\[
\int_{\bG(N,m)}F_V(P)\,d\rho(P) = \int_{\bR^N\times\bG(N,m)} \Bigl[ \int_{\bG(N,m)} \phi (E_P^tB_\Psi(T)E_P)\,d\rho(P) \Bigr] \, dV(x,T) \geq \kappa \|V\|(\bR^N).
\]
Therefore, there exists some $P\in\bG(N,m)$ such that $F_V(P) \geq  \kappa\, \|V\|(\bR^N)$.
Applying Lemma~\ref{lemma:concavity-AP}, we conclude that $\Phi(\mathbf A_P) \geq F_V(P) \geq \kappa \|V\|(\bR^N)$ as claimed.
\end{proof}

For Theorem~\ref{thm:anisotropic-MS}, we will specialize the above result to $\phi = \lambda_m$ and take $\rho$ to be the normalized Haar measure on $\bG(N,m)$.
We then define, for fixed $T \in \bG(N,m)$,
\begin{equation}\label{eqn:d-N-m-beta}
    a_{N,m} := \int_{\mathbb{G}(N,m)} \lambda_{m}(E^t_P T E_P) \, dP, \qquad T \in \mathbb{G}(N,m).
\end{equation}
This quantity is independent of the choice of $T$ because the action of $O(N)$ on $\mathbb{G}(N,m)$ is transitive, while the functional $\lambda_{m}$ is invariant under orthogonal conjugation.
\begin{lemma}\label{lemma:aN,m}
    For every $m<N$, we have $a_{N,m} \geq \binom{N}{m}^{-1}> 0$.
\end{lemma}
\begin{proof}
    By rotational invariance, we can fix $T \in \bG(N,m)$ to be $\textup{span} \{ e_1, \dots, e_m \}$.
    For $P\in \bG(N,m)$, let $Q \in \bR^{N\times m}$ have orthogonal columns spanning $P$ and satisfy $Q^t Q =I_m$.
    Then, $E^t_P TE_P = Q^t TQ$ is a symmetric positive semidefinite matrix whose eigenvalues all lie in $[0,1]$.
    Denoting these by $0 \leq \lambda_m \leq \cdots \leq \lambda_1 \leq 1$, we obtain
    \[
    \det (Q^t TQ) = \prod_{i=1}^m \lambda_i \leq \lambda_m = \lambda_m (Q^t T Q), \qquad \implies \qquad a_{N,m} \geq \int_{\bG(N,m)} \det (Q^t TQ) \, dP.
    \]
    Let $Q_I$ denote the $m \times m$ sub-matrix of $Q$ obtained by retaining the rows indexed by an $m$-element subset $I \subset \{ 1, \dots, N \}$.
    The Cauchy-Binet identity together with $Q^t Q = I_m$ shows that 
    \[
    1 = \det (Q^t Q) = \sum_{ |I|=m } \det (Q_I)^2.
    \]
    On the other hand, for $I_0 = \{ 1, \dots, m \}$, the choice of $T$ gives $\det(Q^t TQ) = \det (Q^t_{I_0} Q_{I_0}) = \det (Q_{I_0})^2$.
    Integrating over $\bG(N,m)$, Haar invariance implies that $\int_{\bG(N,m)} \det (Q_I)^2 \, dP$ has the same value for every $m$-element subset $I \subset \{ 1, \dots, N \}$.
    Since there are $\binom{N}{m}$, we obtain
    \[
    1 = \binom{N}{m} \int_{\bG(N,m)} \det (Q_I)^2 \, dP, \qquad \implies \qquad \int_{\bG(N,m)} \det (Q^t TQ) \, dP = \binom{N}{m}^{-1}.
    \]
    The claimed bound $a_{N,m} \geq \binom{N}{m}^{-1}$ follows from combining the preceding inequalities.
\end{proof}

\section{Quantitative non-concentration and the vanishing mass theorem}\label{section:measure-theory}

The proof of Theorem~\ref{thm:anisotropic-MS} will invoke some key tools from measure theory.
We first record a standard strict approximation lemma for vector-valued measures.

\begin{lemma}\label{lemma:strict-approximation}
    For $n, d \geq 1$, consider a measure $\mu \in \cM( \bR^n; \bR^d)$ and let $\rho_{\ve}$ be a standard non-negative mollifier, with $\mu_{\ve} := \rho_{\ve} \ast \mu \in \cL^n$.
    Then, $\mu_{\ve} \to \mu$ strictly as $\ve \downarrow 0$, meaning that
    \[
    \mu_{\ve} \xrightharpoonup{*} \mu, \qquad |\mu_{\ve}|(\bR^n) \to |\mu|(\bR^n).
    \]
    More generally, let $\cA = \sum_{j=1}^n A_j \partial_j$ be a first-order linear operator and suppose that $\cA \mu \in \cM ( \bR^n; \bR^{\ell})$.
    Then, we have $\cA \mu_{\ve} = ( \rho_{\ve} \ast \cA \mu) \, \cL^n$ and $\cA \mu_{\ve} \to \cA \mu$ strictly.
    If the measures $\mu$ and $\cA \mu$ are compactly supported, both families of variation measures are uniformly tight.
\end{lemma}
\begin{proof}
    The weak-* convergence is standard because $\{ \rho_{\ve} \}$ is an approximate identity.
    By Jensen's inequality, Fubini's theorem, and the weak-* lower semicontinuity of the total variation,
    \[
    |\mu_{\ve}|(\bR^n) = \int_{\bR^n} |\rho_{\ve} \ast \mu| \, dx \leq |\mu| (\bR^n) \leq \liminf_{\ve \downarrow 0} |\mu_{\ve}|(\bR^n),
    \]
    and hence $\mu_{\ve} \to \mu$ strictly; this proves the first assertion.
    Next, constant coefficient differential operators commute with convolution, so $\cA \mu_{\ve} = ( \rho_{\ve} \ast \cA \mu) \, \cL^n$.
    Hence, the same result applies to $\cA \mu$.
    Finally, for a compactly supported standard mollifier with $\textup{spt} \, \rho_{\ve} \subset B_{C \ve}$, we have
    \[
    \textup{spt} \, \mu_{\ve} \subset \textup{spt} \, \mu + B_{C \ve}, \qquad \textup{spt} \, \cA \mu_{\ve} \subset \textup{spt} \, \cA \mu + B_{C \ve}.
    \]
    Thus, the above measures are uniformly tight.
    This completes the proof.
\end{proof}

Under the strict convergence of measures as above, we can invoke Reshetnyak's continuity theorem (see, for example~\cite{reshetnyak} and~\cite{ambrosio-fusco-pallara}*{Theorem~2.39}) in the following form.
\begin{lemma}\label{lemma:reshetnyak-continuity}
    Let $\mu_j, \mu \in \cM( \Omega; \bR^d)$ be vector-valued measures on an open set $\Omega \subset \bR^n$ that converge strictly, $\mu_j \to \mu$, meaning that $\mu_j \xrightharpoonup{*} \mu$ and $|\mu_j|(\Omega) \to |\mu|(\Omega)$.
    Let $F: \Omega \times \bR^d \to \bR$ be a continuous, positively $1$-homogeneous function in the second variable satisfying a uniform linear bound,
    \[
    F(x,tz) = t \, F(x,z) \quad \text{for every } \; t \geq 0, \qquad |F(x,z)| \leq C |z| \quad \text{for every } \; (x,z) \in \Omega \times \bR^d.
    \]
    Then, it holds that
    \[
    \int_{\Omega} F \Bigl( \frac{d \mu_j}{d |\mu_j|} \Bigr) \, d |\mu_j| \to \int_{\Omega} F \Bigl( \frac{d \mu}{d |\mu|} \Bigr) \, d |\mu|.
    \]
    The same statement holds on the flat torus $\bT^n$.
\end{lemma}

We will also require some results from the resolution of the vanishing mass theorem~\cite{vanishing-mass-conjecture} of Gennaioli and Rindler.
Their proof involves a smoothing procedure using the heat flow, with $\cP_T$ the heat semigroup.
A key step is their compensated compactness result~\cite{vanishing-mass-conjecture}*{Proposition 3.5}, whose estimates are weighted with the function $a_T := \cP_T \mathbf{1}_D$, for $D = (s,1-s)^m$ an $m$-dimensional cube, where $s \in (0,\frac{1}{2})$.
This result can be modified to an unweighted estimate.

\begin{lemma}\label{lemma:unweighted-GR-step}
Let $\phi$ be a concave barrier and consider a smooth function $A \in C_c^{\infty}(\bR^m ; \bR^{m \times m})$.
For a given $T \in (0,1]$, we suppose that
\[
\| \cP_T f \|_{L^{\infty}} < R_0 < e^{-1}R, \qquad \text{where } \; f := |A| + \sqrt{T} \, |\textup{Div} \, A|.
\]
We denote $N_{\phi} := \max_{ |Z| \leq 1 } \phi(Z)_+$.
Then, for every $\delta>0$, it holds that
\[
\int_{\bR^m} \phi(A) \, dx \leq \delta \int_{\bR^m} f + N_{\phi} \left( \frac{C(m,\phi,\delta)}{\log(R/R_0)} \int_{\bR^m} f + \int_{ \{ f \leq R \} } f \right).
\]
\end{lemma}
\begin{proof}
We write $q := - \phi$.
As observed above, $q$ is convex, positively one-homogeneous, and non-negative on the wave cone $\Lambda_{\textup{Div}}$ for row-wise divergence, so~\cite{vanishing-mass-conjecture}*{Proposition 3.5} applies, first in weighted form.
We fix once and for all a parameter $s \in (0, \frac{1}{2})$ and write $D = (s, 1-s)^m$ for the cube appearing in their result, with side length $\ell := 1-2s$.
For $k \in \bZ^m$, we define $D_k := D + \ell k$ and $a_{T,k} := \cP_T \mathbf{1}_{D_k}$.
Because~\cite{vanishing-mass-conjecture}*{Proposition 3.5} is translation invariant, it applies, after a change of variables, to every translated cube $D_k$ and $A_k(x) := A(x + \ell k)$ in place of $D, A$, with the same constants and the same numbers $R_0,R$.
For brevity, we write $M_k := \int a_{T,k} f$ and $M_{R,k} := \int_{ \{ f \leq R \} } a_{T,k} f$, so this result gives
\[
\int a_{T,k} q(A) \geq - \delta M_k - N_q \Bigl( \frac{C(m,\phi,\delta)}{\log(R/R_0)} M_k + M_{R,k} \Bigr)
\]
with constant $N_q = \max_{ |Z| \leq 1 } \{ -q(Z),0\} = N_{\phi}$.
Since $q$ has linear growth and the functions $A,f \in L^1$, all the relevant sums are absolutely convergent and satisfy
\begin{equation}\label{eqn:sum-atk-q(A)}
\sum_k \int a_{T,k} \, |q(A)| = \int \Bigl( \sum_k a_{T,k} \Bigr) \, |q(A)| \leq C_{\phi} \int |A| < \infty
\end{equation}
for a constant $C_{\phi}$.
The same argument applies to the non-negative numbers $M_k, M_{R,k}$.
The cubes $D_k$ partition $\bR^m$ up to a null set, so the positivity and linearity of the heat semigroup imply
\begin{equation}\label{eqn:sum-atk-heat-semigroup}
    \sum_{k \in \bZ^m} a_{T,k} = \cP_T \Bigl( \sum_{k \in \bZ^m} \mathbf{1}_{D_k} \Bigr) = 1.
\end{equation}
We sum over $k$ and combine the properties~\eqref{eqn:sum-atk-q(A)} and~\eqref{eqn:sum-atk-heat-semigroup}, together with 
\[
\sum_k M_k = M, \qquad \sum_k M_{R,k} = M_R, \qquad \sum_k \int a_{T,k} q(A) = \int q(A).
\]
Our assertion follows from summing the weighted inequalities~\eqref{eqn:sum-atk-q(A)} and substituting $q=-\phi$.
\end{proof}

\begin{proposition}\label{prop:GR-barrier-compactness}
Let $\phi$ be a concave barrier and let $\Phi$ be its associated function on vector-valued measures, defined in~\eqref{eqn:Lambda-m-definition}.
Consider a sequence of measures $\mathbf{A}_j \in \cM( \bR^m ; \bR^{m \times m}) $ with finite row-wise divergence such that $\sup_j |\mathbf{A}_j| ( \bR^m ) < \infty$, $|\textup{Div} \, \mathbf{A}_j| (\bR^m) \to 0$, and each $\mathbf{A}_j$ is concentrated on a Borel set $E_j \subset \bR^m$ with $\cL^m(E_j) \to 0$.
Then, we have $\limsup_{j \to \infty} \Phi( \mathbf{A}_j) \leq 0$.
\end{proposition}
\begin{proof}
Since $\phi$ is continuous and positively one-homogeneous, there is a
constant $C_\phi<\infty$ such that $|\phi(A)|\leq C_\phi|A|$.
The proof proceeds in a number of steps, ultimately reducing the statement to an application of~\cite{vanishing-mass-conjecture}*{Proposition 3.5} in the unweighted form obtained in Lemma~\ref{lemma:unweighted-GR-step}.
We note that Proposition~\ref{prop:GR-barrier-compactness} could also be obtained through an argument by contradiction, after convolving with a periodic mollifier and reducing the statement to~\cite{vanishing-mass-conjecture}*{Proposition 3.1}.

\smallskip \noindent \textbf{Step 1.}
We first reduce the statement to compactly supported measures: choose a sequence of radii $R_j \to \infty$ so that $|\mathbf{A}_j|( \bR^m \setminus B_{R_j}) \leq j^{-1}$ and let $\chi_j \in C_c^{\infty}(\bR^m)$ be cutoff functions with $0 \leq \chi_j \leq 1$, $\chi_j = 1$ on $B_{R_j}$ and $\chi_j = 0$ outside $B_{2R_j}$, and with $|\nabla \chi_j| \leq C R_j^{-1}$.
We now consider the measures $\tilde{\mathbf{A}}_j := \chi_j \mathbf{A}_j$.
Multiplication by a non-negative scalar function preserves the polar direction, hence
\begin{equation}\label{eqn:cutoff-Phi}
|\Phi( \tilde{\mathbf{A}}_j) - \Phi( \mathbf{A}_j)| \leq C_{\phi} |\mathbf{A}_j|(\bR^m \setminus B_{R_j}) \leq C_{\phi} j^{-1}.
\end{equation}
Moreover, the product rule gives $\textup{Div} \, \tilde{\mathbf{A}}_j = \chi_j \, \textup{Div} \, \mathbf{A}_j + \mathbf{A}_j\nabla \chi_j$, and hence
\[
|\textup{Div} \, \tilde{\mathbf{A}}_j| (\bR^m) \leq |\textup{Div} \, \mathbf{A}_j| (\bR^m) + C R_j^{-1} \, |\mathbf{A}_j|(\bR^m) \to 0 \qquad \text{as } \; j \to \infty.
\]
Finally, the measure $\tilde{\mathbf{A}}_j$ is concentrated on the set $E_j \cap B_{2R_j}$, whose Lebesgue measure tends to zero as $j \to \infty$.
In view of~\eqref{eqn:cutoff-Phi}, we may therefore replace
$\mathbf{A}_j$ by $\tilde{\mathbf{A}}_j$ and assume from now on that the measures
are compactly supported.

\smallskip \noindent \textbf{Step 2.}
Let $\rho_{\eta}$ be a standard compactly supported mollifier.
We claim that, for every $j$, we can choose a sufficiently small $0 < \eta_j < j^{-1}$ so that the measure $A_j := \rho_{\eta_j} \ast \mathbf{A}_j$ satisfies
\begin{equation}\label{eqn:phi(Aj)-Ej}
    \Bigl| \Phi( \mathbf{A}_j) - \int_{\bR^m} \phi(A_j) \, dx \Bigr| \leq j^{-1} \qquad \text{and} \qquad \bigl| \{ x \in E^c_j : |A_j(x)| > j^{-1}\} | \leq j^{-1}.
\end{equation}
By the inner regularity of Radon measures, we can first choose a compact set $K_j \subset E_j$ such that $|\mathbf{A}_j| (\bR^m \setminus K_j) \leq j^{-1}$.
Using the strict convergence under convolution $\rho_{\eta} \ast \mathbf{A} \to \mathbf{A}$ as $\eta \downarrow 0$ for compactly supported measures from Lemma~\ref{lemma:strict-approximation}, as well as Reshetnyak's continuity theorem~\ref{lemma:reshetnyak-continuity} for the continuous, positively one-homogeneous function $\phi$, we can produce a sufficiently small $0 < \eta_j < j^{-1}$ so that the first part of~\eqref{eqn:phi(Aj)-Ej} is satisfied.
The second property follows after decreasing $\eta_j$ further and using the upper continuity of the Lebesgue measure for the decreasing compact neighborhoods $K_j + B_{\eta}$, which contain $\textup{spt} ( \rho_{\eta} \ast (\mathbf{A}_j \mres K_j))$.
Applying the Lebesgue-Besicovitch differentiation theorem~\cite{maggi}*{Theorem~5.8} to the finite vector-valued measures $\mathbf{A}_j$, we also see that $(\rho_{\ve} \ast \mathbf{A}_j)(x) \to \frac{d \mathbf{A}_j}{d \cL^m}(x)$ as $\eta \downarrow 0$, for $\cL^m$-a.e.~$x$.
Since $\mathbf{A}_j$ is concentrated on $E_j$, the density on the right-hand side vanishes for a.e.~$x \in E^c_j$.
This proves~\eqref{eqn:phi(Aj)-Ej}.

\smallskip \noindent \textbf{Step 3.}
We now decompose the smooth function $A_j$ as
\[
A_j := \rho_{\eta_j} \ast ( \mathbf{A}_j \mres K_j) + \rho_{\eta_j} \ast ( \mathbf{A}_j \mres (\bR^m \setminus K_j) ) =: \tilde{A}_j + W_j.
\]
The first term is supported on the compact region $\tilde{K}_j := K_j + B_{\eta_j}$, with $\cL^m( \tilde{K}_j) \to 0$ and
\[
\| W_j \|_{L^1} \leq |\mathbf{A}_j| (\bR^m \setminus K_j ) \leq j^{-1}, \qquad \cL^m ( \tilde{K}_j) \leq \cL^m(K_j) + j^{-1} \leq \cL^m(E_j) + j^{-1} \to 0
\]
by the above construction~\eqref{eqn:phi(Aj)-Ej}.
Moreover, for every fixed $\delta>0$ and sufficiently large $j$, we have
\begin{align*}
j^{-1} < \delta & \implies \{ |A_j| > \delta \} \subset E_j \cup \{ x \in E^c_j : |A_j(x)| > j^{-1}\}, \\
& \implies | \{ |A_j| > \delta \} | \leq |E_j| + j^{-1} \to 0.
\end{align*}
Likewise, for every fixed $R< \infty$, we use $\tilde{A}_j = 0$ outside $\tilde{K}_j$, so $A_j = W_j$ there, to find
\begin{equation}\label{eqn:Aj-leq-R-goes-to-zero}
\int_{ \{ |A_j| \leq R \} } |A_j| \, dx \leq R \cL^m ( \tilde{K}_j) + \|W_j \|_{L^1} \to 0.
\end{equation}
In particular, $A_j \to 0$ in measure.
Finally, we obtain
\begin{equation}\label{eqn:div-goes-to-zero}
\textup{Div} \, A_j = \rho_{\eta_j} \ast ( \textup{Div} \, \mathbf{A}_j), \qquad \implies \qquad \| \textup{Div} \, A_j \|_{L^1} \leq |\textup{Div} \, \mathbf{A}_j|(\bR^m) \to 0
\end{equation}
because convolution commutes with row-wise divergence.

\smallskip \noindent \textbf{Step 4.}
We now complete the proof.
We will apply Lemma~\ref{lemma:unweighted-GR-step} with $T=1$.
For this, we set
\[
f_j := |A_j| + |\textup{Div} \, A_j|, \qquad M_j := \int_{\bR^m} f_j, \qquad M_{j,R} := \int_{ \{ f_j \leq R \} } f_j.
\]
The numbers $M_j$ are uniformly bounded, $M_* := \sup_j M_j < \infty$, and satisfy
\begin{align*}
M_j &= \|A_j \|_{L^1} + \| \textup{Div} \, A_j \|_{L^1} \leq |\mathbf{A}_j|(\bR^m) + |\textup{Div} \, \mathbf{A}_j|(\bR^m), \\
M_{j,R} &\leq \int_{ \{ |A_j| \leq R \} } |A_j| + \| \textup{Div} \, A_j \|_{L^1} \to 0
\end{align*}
in view of the properties~\eqref{eqn:Aj-leq-R-goes-to-zero} and~\eqref{eqn:div-goes-to-zero}.
Also, $\| \cP_1 f_j \|_{L^{\infty}} \leq ( 4\pi)^{- \frac{m}{2}} M_j$ by the
standard $L^1-L^{\infty}$ estimate for the Euclidean heat kernel; see, for example,~\cite{cheng-li-yau}.
Thus, we can choose $R_0$ sufficiently large, once and for all, so that $R_0 > ( 4\pi)^{- \frac{m}{2}} \sup_j M_j$ for every $j$, and hence $\| \cP_1 f_j \|_{L^{\infty}} < R_0$.
Finally, we can apply Lemma~\ref{lemma:unweighted-GR-step} with $T=1$: for fixed $\delta>0$ and $R>eR_0$, we obtain
\[
\int \phi(A_j) \leq \delta M_j + N_{\phi} \Bigl( \frac{C(m,\phi, \delta)}{\log (R/R_0)} M_j + M_{j,R} \Bigr).
\]
We first send $j \to \infty$ and use $M_* := \sup_j M_j < \infty$.
For fixed $\delta>0$ and $R>eR_0$, after taking $\limsup_{j \to \infty}$, the tail property $\limsup_{j \to \infty} M_{j,R} = 0$ implies that
\[
\limsup_{j \to \infty} \int \phi(A_j) \, dx \leq \delta M_* + N_{\phi} C(m,\phi,\delta) \, \log(R/R_0)^{-1} \, M_*.
\]
In the next step, we send $R \to \infty$, followed by $\delta \downarrow 0$, to obtain
\[
\limsup_{j \to \infty} \int \phi(A_j) \, dx \leq \delta M_*, \qquad \overset{\delta \downarrow 0}{\implies} \qquad \limsup_{j \to \infty} \int \phi(A_j) \, dx \leq 0.
\]
Together with the property~\eqref{eqn:phi(Aj)-Ej}, this completes the proof of our assertion.
\end{proof}

We now use Proposition~\ref{prop:GR-barrier-compactness} to prove a quantitative non-concentration theorem for vector-valued measures.
Note that the matrix-valued measures below are not assumed compactly supported.
\begin{corollary}\label{cor:concave-barrier}
Let $\phi$ be a concave barrier with associated function $\Phi$ on measures, defined in~\eqref{eqn:Lambda-m-definition}.
For every $L < \infty$ and $\eta>0$, there is a $c = c(m,\phi,L,\eta)>0$ with the following property.

Consider some $\mathbf{A} \in \cM( \bR^m ; \bR^{m \times m})$ with finite row-wise divergence, concentrated on a Borel set $E$ of finite Lebesgue measure, and satisfying $|\mathbf{A}|(\bR^m) \leq L$ and $\Phi(\mathbf{A}) \geq \eta$.
Then, it holds that 
\[
\cL^m(E)^{\frac{1}{m}} |\textup{Div} \, \mathbf{A}| (\bR^m) \geq c.
\]
\end{corollary}
\begin{proof}
We argue by contradiction.
Suppose that this property fails, so there exists some $L < \infty$ and $\eta>0$ along with a sequence of measures $\mathbf{A}_j$ supported on Borel sets that satisfy the bounds $|\mathbf{A}_j|(\bR^m) \leq L$ and $\Phi( \mathbf{A}_j) \geq \eta$, but have
\[
e_j d_j \to 0 , \qquad \text{where}\qquad  e_j := \cL^m(E_j)^{\frac{1}{m}} \quad\text{and}\quad d_j := |\textup{Div} \, \mathbf{A}_j|(\bR^m).
\]
We choose a sequence of radii $r_j>0$ such that $r_j e_j \to 0$ and $r_j^{-1} d_j \to 0$.
When $e_j, d_j> 0$, we can take $r_j = (d_j/ e_j)^{\frac{1}{2}}$; if $e_j d_j = 0$ for all but finitely many members of the sequence, we let
\[
r_j := j(1+d_j) \quad \text{when } \; e_j = 0, \qquad r_j := [ j(1+e_j)]^{-1} \qquad \text{when } \; d_j = 0 < e_j.
\]
We dilate the measures $\mathbf{A}_j$ by the homotheties $D_{r_j}(x) := r_j x$ and let $\tilde{\mathbf{A}}_j := (D_{r_j})_{\#} \mathbf{A}_j$.
Since $D_r$ is a bijection and $(D_{r_j})_{\#}$ is the ordinary pushforward of a matrix-valued measure, we find 
\begin{equation}\label{eqn:pushforwards}
| (D_{r_j})_{\#} \mathbf{A}_j| = (D_{r_j})_{\#} |\mathbf{A}_j|, \qquad |(D_{r_j})_{\#} ( \textup{Div} \, \mathbf{A}_j)| = (D_{r_j})_{\#} | \textup{Div} \, \mathbf{A}_j|
\end{equation}
Moreover, a direct computation in the definition~\eqref{eqn:divergence-definition}, testing with rescaled vector fields $X$ and using the change of variables formula, shows that
\begin{equation}\label{eqn:divergence-drifting}
    \textup{Div} \, \tilde{\mathbf{A}}_j = r_j^{-1} (D_{r_j})_{\#} (\textup{Div} \, \mathbf{A}_j) \quad 
    \implies \quad | \textup{Div} \, \tilde{\mathbf{A}}_j |(\bR^m) = r_j^{-1} |\textup{Div} \, \mathbf{A}_j|(\bR^m) = r_j^{-1} d_j \to 0
\end{equation}
using property~\eqref{eqn:pushforwards}.
This fact also shows that the total mass is preserved under pushforwards, and the positive one-homogeneity of $\phi$ implies that $\Phi$ is preserved as well.
Therefore, 
\[
|\tilde{\mathbf{A}}_j|(\bR^m) = |\mathbf{A}_j|(\bR^m) \leq L \qquad \text{and} \qquad \Phi( \tilde{\mathbf{A}}_j) = \Phi( \mathbf{A}_j) \geq \eta.
\]
Finally, the rescaled measures are supported on $D_{r_j}(E_j)$, where $\cL^m(D_{r_j}(E_j))^{\frac{1}{m}} = r_je_j \to 0$.
Thus, they satisfy the conditions of Proposition~\ref{prop:GR-barrier-compactness}, which implies $\limsup_{j \to \infty} \Phi( \tilde{\mathbf{A}}_j) \leq 0$; this contradicts $\Phi( \tilde{\mathbf{A}}_j) \geq \eta>0$,  thereby proving our assertion.
\end{proof}

\section{Proof of the main theorem}

We now prove Theorem~\ref{thm:anisotropic-MS}.
As discussed in the introduction, our tools lead to a much more general result for every anisotropy satisfying a spectral condition that is open in the $C^1$ topology.
We will formulate and prove this result as follows.
For this, we recall the notation $\kappa(\Psi, \Gamma, \phi, \rho)$ introduced in~\eqref{eqn:kappa-Psi-Gamma-phi-rho}, for $\Psi, \Gamma, \phi, \rho$ respectively an anisotropy, a Borel set $\Gamma \subset \bG(N,m)$, a concave barrier, and a probability measure on $\bG(N,m)$.

\begin{theorem}\label{thm:most-general-michael-simon}
Consider a Borel set $\Gamma \subset \bG(N,m)$.
For every $\kappa_* > 0$ and $B_* < \infty$, there is a constant $C(N,\phi,\kappa_*, B_*)$ so that for every anisotropy with
\begin{equation}\label{eqn:open-condition}
\textstyle \kappa(\Psi,\Gamma,\phi,\rho) > \kappa_* > 0, \qquad \sup_{T \in \Gamma} \|B_{\Psi}(T) \| < B_*< \infty,
\end{equation}
every rectifiable $m$-varifold $V = v(M,\theta)$ with $T_x M \in \Gamma$ for $\cH^m$-a.e.~$x$ satisfies
\[
\| V\| (\bR^N) \leq C \, \cH^m(M)^{\frac{1}{m}} |\delta_{\Psi}V |(\bR^N).
\]
\end{theorem}

The condition~\eqref{eqn:open-condition} is open in the $C^1$ topology.
Indeed, the linear bound~\eqref{eqn:linear-bound} for $\phi$, combined with the property that projection with $E_P$ is non-increasing in operator norm, yields
\begin{align*}
& \Bigl| \int \phi(E_P^t B_\Psi(T)E_P)\,d\rho(P) - \int \phi(E_P^t B_{\tilde\Psi}(T)E_P)\,d\rho(P) \Big| \leq C_\phi \sup_{S\in\bG(N,m)} \|B_\Psi(S)-B_{\tilde\Psi}(S)\|.
\end{align*}
For $\tilde{\Psi}$ $C^1$-close to $\Psi$, the definition~\eqref{eqn:BPsi(T)} of $B_\Psi$ shows that $\sup_{T \in \bG(N,m)} \|B_{\tilde{\Psi}}(T) - B_{\Psi}(T) \| \leq 2 \, \| \tilde{\Psi} - \Psi\|_{C^1}$
Therefore, both inequalities in~\eqref{eqn:open-condition} are stable under $C^1$ perturbations of $\Psi$.

\begin{proof}[Proof of Theorem~\ref{thm:most-general-michael-simon}]
Let us abbreviate $\kappa := \kappa(\Psi, \Gamma, \phi, \rho) > \kappa_*$.
Using Lemma~\ref{lem:averagedCoercivityGeneral}, we obtain an $m$-plane $P \in \bG(N,m)$ such that $\Phi( \mathbf{A}_P) \geq \kappa \|V\|(\bR^N) > \kappa_* \|V\|(\bR^N)$.
We define a rescaled vector-valued measure $\hat{\mathbf{A}} := \|V\|(\bR^N)^{-1} \mathbf{A}_P$, so $\Phi( \hat{\mathbf{A}}) > \kappa_*$ by homogeneity.
Using Lemma~\ref{lemma:A-P-projection-properties}, $T_x M \in \Gamma$, and $\sup_{T\in\Gamma}\|B_\Psi(T)\| <B_*$, we can control
\[
|\on{Div}\widehat{\mathbf{A}}|(\bR^m) \leq \|V\|(\bR^N)^{-1}|\delta_\Psi V|(\bR^N), \qquad |\widehat{\mathbf{A}}|(\bR^m) \leq C_m B_*.
\]
Moreover, $\widehat{\mathbf{A}}$ is concentrated on a Borel set
$G_P\supset\pi_P(M)$ with $\cL^m(G_P) = \cL^m(\pi_P(M)) \leq \cH^m(M)$.
We can now apply Corollary~\ref{cor:concave-barrier} with $\eta=\kappa_*$ and $L = C_m B_*$ to produce $c(m,\phi,B_*,\kappa_*)>0$ with
\[
c \leq \cL^m(G_P)^{\frac1m} |\on{Div}\widehat{\mathbf{A}}|(\bR^m) \leq \cH^m(M)^{\frac1m} \|V\|(\bR^N)^{-1} |\delta_\Psi V|(\bR^N).
\]
Our assertion follows by rearranging terms.
\end{proof}
\begin{proof}[Proof of Theorem~\ref{thm:anisotropic-MS}]
This result follows from Theorem~\ref{thm:most-general-michael-simon} after specializing to $\Gamma = \bG(N,m)$, $\rho = dP$ the Haar measure, and $\phi = \lambda_m$.
Indeed, using the $1$-Lipschitz property of $\lambda_m$ in the definition~\eqref{eqn:kappa-Psi-Gamma-phi-rho} of $\kappa$, together with the definition~\eqref{eqn:d-N-m-beta} of the constant $a_{N,m}$, we obtain
\[
\kappa ( \Psi, \bG(N,m), \lambda_m , dP) \geq a_{N,m} - \sup_{T \in \bG(N,m)} \| B_{\Psi}(T) - T \|
\]
where $a_{N,m} \geq \binom{N}{m}^{-1}$.
Taking $\ve_{N,m} = \frac{1}{2} a_{N,m}$ therefore allows us to apply Theorem~\ref{thm:most-general-michael-simon} whenever $\sup \|B_{\Psi}(T) - T \| \leq \ve_{N,m}$, proving the first bound for $C(N,\lambda_m, \ve_{N,m}, B_{N,m}) = C_N$ a dimensional constant.
If $\Theta^m( \|V\|, x) \geq \theta_0$ for $\|V\|$-a.e.~$x$, then $\cH^m (M) \leq \theta_0^{-1} \|V\|(\bR^N)$ and $\cH^m(M)^{\frac{1}{m}} \leq \theta_0^{- \frac{1}{m}} \|V\|(\bR^N)^{\frac{1}{m}}$.
The second inequality follows from substituting this bound above.
\end{proof}
The constant $a_{N,m}$ is computable using Selberg integrals; see, for example,~\cite{dtmf}*{Ch.~35}.
In~\cite{anisotropic-surfaces}*{Lemma~4.2}, we compute $a_{N,2}$ explicitly and prove that $a_{N,2} > \frac{4-\pi}{2N}$ for every $N \geq 3$.

\begin{remark}\label{rmk:jensen}
We collect here some remarks on Theorem~\ref{thm:most-general-michael-simon}.
\begin{enumerate}[$(i)$]
    \item The freedom to consider Borel subsets $\Gamma \subset \bG(N,m)$ in Theorem~\ref{thm:most-general-michael-simon} implies Michael-Simon inequalities for varifolds whose tangent plane map $x \mapsto T_x V$ is confined inside $\bG(N,m)$.
For example, the result holds for varifolds whose tangent planes lie $\|V\|$-a.e.~ in $\Gamma \subset \bG(N,m)$ with 
\[
\tau := \inf_{T \in \Gamma} \min_{p \in P_0 \cap \bS^{N-1}} |\pi_T p|^2 > 0, \qquad P_0 \in \bG(N,m) \quad \text{a fixed $m$-plane}
\]
and every anisotropy $\Psi$ such that $\sup_{T \in \Gamma} \|B_{\Psi}(T) - T \| <\tau$.
In particular, one can take $\Gamma$ to be the family of $m$-planes that are linear graphs with uniformly bounded slope over $P_0$.
    \item By Jensen's inequality, it is clear that Theorem~\ref{thm:most-general-michael-simon} continues to hold whenever there exists a collection of concave barriers $\{ \phi_\alpha \}_{\alpha=1}^J$, planes $\{ P_\alpha \}_{\alpha=1}^J \in \bG(N,m)$, and isometries $E_\alpha: \bR^m \to P_\alpha$ such that $\inf_{T \in \Gamma} \sum_{\alpha=1}^J \phi_\alpha (E^t_\alpha B_{\Psi}(T) E_\alpha) > 0$.
    \item By item $(ii)$, the results of Section~\ref{section:measure-theory} extend to matrix-valued measures in $\bR^{N \times m}$, with concave barriers satisfying $\phi(A) \leq 0$ on
    \[
    \Lambda_{\textup
    Div} = \{ A : \exists \; \xi \neq 0 : A \xi = 0 \} = \{ A : \textup{rank} \, A < m \},
    \]
    which is the wave cone for row-wise divergence on these matrices.
    The results of~\cite{vanishing-mass-conjecture}*{\S 3} apply to general first-order operators with constant coefficients, so Corollary~\ref{cor:concave-barrier} transfers directly to $N \times m$-matrix-valued measures.
    Denoting by $\Phi(\mathbf{A})$ the associated function on matrix-valued measures, defined as in~\ref{eqn:Lambda-m-definition}, we again obtain, as in Corollary~\ref{cor:concave-barrier}, 
    \begin{equation}\label{eqn:quantitative-compactness}
        |\mathbf{A}|(\bR^m) \leq L, \quad \Phi(\mathbf{A}) \geq \eta>0, \quad \implies \quad \cL^m(E)^{\frac{1}{m}} |\textup{Div}\, \mathbf{A}(\bR^m)| \geq c(N,m,\phi,L,\eta)>0.
    \end{equation}
    Consequently, the result of item $(ii)$ also holds for restrictions $B_{\Psi}(T)|_P : P \to \bR^N$ of the stress tensor, with corresponding $\bR^{N \times m}$-valued measure $\mathbf{A}_P := (\pi_P)_{\#} [ B_{\Psi}(T)|_P \, dV(x,T)]$. 
\end{enumerate}
\end{remark}

We now prove Theorem~\ref{thm:most-important-application} and discuss interesting applications of Theorem~\ref{thm:most-general-michael-simon}.
\begin{proof}[Proof of Theorem~\ref{thm:even-convex}]
Let $\Psi_0 \in C^1( \bG(m+1,m), (0,\infty))$ have convex even one-homogeneous extension $F_0( \nu) = |\nu| \, \Psi_0(\nu^{\perp})$ with $F_0(0)=0$.
For each hyperplane $P$, we define
\[
\phi_P(A) := \inf_{ P \cap \{ F_0(z) = 1 \} } dF_0(z) [Az], \qquad A \in \textup{Hom}(P, \bR^N).
\]
This is a concave, positively one-homogeneous Lipschitz function, as the infimum of a family of linear functionals.
If $\textup{rank} \, A<m$, choosing $z \in \ker A$ with $F_0(z)=1$ shows that $\phi_P(A) \leq 0$, so $\phi_P$ is a concave barrier for $\bR^{N \times m}$-valued measures as in Remark~\hyperref[rmk:jensen]{\ref{rmk:jensen}$ \,(iii)$}.
Since $F_0$ is convex, even, and homogeneous, it satisfies $|dF_0(\nu) [y]| \leq F_0(y)$ for $\nu \neq 0$, by the discussion of~\eqref{eqn:F-star-dF-nu}.
Consequently, 
\[
dF_0(z) [ B_{F_0}(\nu) z] = F_0(\nu) - dF_0(\nu)[z] dF_0(z)[\nu] \geq 0
\]
for $F_0(z)=1$.
Thus, $\phi_P( B_{F_0}(\nu)|_P) \geq 0$ for every $P,\nu$.
Because $dF_0(\nu)[\nu] = F_0(\nu) > 0$ for $\nu \in \bS^{N-1}$, $P_{\nu} := \ker dF_0(\nu)$ is a hyperplane, and $B_{F_0}(\nu) z = F_0(\nu) z$ for $z \in P_{\nu}$.
Now, $\phi_{P_{\nu}}(B_{F_0}(\nu)|_{P_{\nu}}) = F_0(\nu) > 0$, so continuity and compactness produce finitely many hyperplanes $\{ P_i \}_{i=1}^J$ such that
\[
\kappa := \min_{\nu \in \bS^{N-1}} \sum_{i=1}^J \phi_i (B_{F_0}(\nu)|_{P_i}) > 0, \qquad \text{for } \; \phi_i := \phi_{P_i}.
\]
Here, positivity follows because every summand is non-negative by $\phi_P(B_{F_0}(\nu)|_P) \geq 0$ for all $P,\nu$, and for every $\nu$, there exists at least one $P_i$ producing strict inequality.
Moreover, given any other anisotropy $\Psi \in C^1(\bG(m+1,m))$ and letting $L_i$ denote the Lipchitz constant of $\phi_i$, we obtain 
\[
\sum_{i=1}^J \phi_i (B_{\Psi}(T)|_{P_i}) \geq \sum_{i=1}^J \phi_i ( B_{\Psi_0}(T)|_{P_i}) - \left( \sum_{i=1}^J L_i \right) \sup_{S \in \bG(m+1,m)} \| B_{\Psi}(S) - B_{\Psi_0}(S) \|.
\]
The latter term satisfies $\sup_S \|B_{\Psi}(S) - B_{\Psi_0}(S) \| \leq 2 \, \| \Psi - \Psi_0 \|_{C^1}$, so we can choose $\ve_0 = \ve_0(N,\Psi_0)$ sufficiently small so that $\sup_{T \in \bG(N,m)} \|B_{\Psi}(T) \| \leq C(N,\Psi_0)$, the above expression is at least $\frac{1}{2} \kappa$, and $\Psi$ remains positive.
The result now follows from items $(ii)$ and $(iii)$ of Remark~\ref{rmk:jensen}.
\end{proof}

\begin{proof}[Proof of Theorem~\ref{thm:most-important-application}]
We follow the framework of~\cite{derosa-tione-regularity}*{\S 6} to describe the $\ell^p$ norm on the Grassmannian $\bG(N,2)$ in terms of the Pl\"ucker embedding by Euclidean simple bivectors, when $p \in (1,\infty)$.
For $T \in \bG(N,2)$, we choose a Euclidean unit simple bivector 
\[
\tau_T = \sum_{1 \leq i < j \leq N} \tau_{ij}\, e_i \wedge e_j, \qquad |\tau|_2 = 1
\]
spanning $T$, and denote by $\Psi_p$ the $\ell^p$ anisotropy, namely 
\begin{equation}\label{eqn:plucker-norm}
    \textstyle \Psi_p(T) := \bigl( \sum_{i<j} |\tau_{ij}|^p \bigr)^{\frac{1}{p}}.
\end{equation}
For each coordinate $2$-plane $P_{ij} := \textup{span} \{ e_i, e_j\}$, let $B_{ij}(T)$ denote the projection of the stress tensor $B_{\Psi}(T)$ onto $P_{ij}$, namely $B_{ij} := E_{ij}^t B_{\Psi}(T) E_{ij}$ under a choice of isometry $E_{ij} : \bR^2 \to P_{ij}$.
By the criterion of Remark~\ref{rmk:jensen}, to prove the result, it suffices to find a concave barrier $\phi_p$ satisfying
\begin{equation}\label{eqn:ell-p-inequality}
    \inf_{T \in \bG(N,2)} \sum_{i<j} \phi_p ( B_{ij}(T)) \geq c_{N,p} > 0
\end{equation}
where $c_{N,p} = \binom{N}{2}^{ \min \{ \frac{1}{p} - \frac{1}{2}, 0 \} }$.
To construct the barrier $\phi_p$, we define $J_p(z) := ( |z_1|^{p-2} z_1, |z_2|^{p-2} z_2)$ for a $2$-vector $z = (z_1, z_2)$ and let
\[
\phi_p (A) := \inf_{ \|z\|_p = 1 } \la Az, J_p( z) \rg, \qquad A \in \bR^{2 \times 2}.
\]
We claim that $\phi_p$ is a concave barrier and the $2$-vector $J_p(z)$ has
\begin{equation}\label{eqn:phip-Jp}
    \la v, w \rg \la J_p(v), J_p(w) \rg \geq 0 \qquad \text{for every } \; v,w \in \bR^2.
\end{equation}
For fixed $\|z\|_p=1$, the map $A \mapsto \la Az , J_p(z) \rg$ is linear, so $\phi_p$ is concave and positively one-homogeneous.  
If $\det A=0$, choosing $z \in \ker A$ with $\|z\|_p = 1$ shows that $\phi_p(A) \leq 0$.
To prove~\eqref{eqn:phip-Jp}, we write $a = v_1 w_1$ and $b = v_2 w_2$, so
\[
\la v,w \rg = a+b, \qquad \la J_p(v) , J_p(w) \rg = |a|^{p-2} a + |b|^{p-2} b.
\]
If $ab \geq0$, the assertion is immediate.
If $ab<0$, both $a+b$ and $|a|^{p-2} a + |b|^{p-2} b$ have the sign of the term with the larger absolute value, since the function $t \mapsto t^{p-1}$ is strictly increasing on $(0,\infty)$, for $p \in (1,\infty)$.
Therefore, $\la v,w \rg \la J_p(v) , J_p(w) \rg \geq 0$ have non-negative product.

Let $\tilde{\Psi}(\tau)$ be the function on $\bigwedge^2 \bR^N$ corresponding to~\eqref{eqn:plucker-norm}, so that $\Psi_p(T) = \tilde{\Psi}(\tau_T)$.
Since $\tilde{\Psi}$ is even, this definition is independent of the choice of orientation for $T$.
We extend the coefficients antisymmetrically by $\tau_{ji} = - \tau_{ij}$ and $\tau_{ii}=0$.
For $L \in \bR^{N \times N}$, we write $L(u \wedge v) := Lu \wedge v+ u \wedge Lv$.
Differentiating the one-homogeneous extension at $t=0$ gives
\begin{equation}\label{eqn:BPsi(T)-L-definition}
    \begin{split}
\frac{d}{dt} \Bigr\rvert_{t=0} \tilde{\Psi}( (I + tL) \tau_T) = \la B_{\Psi_p}(T), L \rg &= \sum_{i,j=1}^N ( B_{\Psi_p}(T))_{ab} L_{ab}, \\\text{where} \qquad (B_{\Psi_p}(T))_{ij} &:= \tilde{\Psi}(\tau)^{1-p} \sum_{k=1}^N |\tau_{ik}|^{p-2} \tau_{ik} \tau_{jk}.
    \end{split}
\end{equation}
Here, we used the fact that $D \tilde{\Psi}(\tau) [ \zeta] = \tilde{\Psi}(\tau)^{1-p} \sum_{i<j} J_p (\tau_{ij}) \zeta_{ij}$ and changed the indices and order of summation.
Notice that~\eqref{eqn:BPsi(T)-L-definition} recovers $B_{\Psi_2}(T) = T$ for $p=2$.
The expression~\eqref{eqn:BPsi(T)-L-definition} implies that the projected stress $B_{ij}(T)$ onto $P_{ij} = \textup{span} \{ e_i, e_j \}$ satisfies
\begin{equation}\label{eqn:projected-stress}
    B_{ij}(T) = \tilde{\Psi}(\tau)^{1-p} \Bigl( |\tau_{ij}|^p I_2 + \sum_{ k \not\in \{ i,j \} } J_p (v^{ij}_k) \otimes v^{ij}_k \Bigr)
\end{equation}
where we denoted $v^{ij}_k := (\tau_{ik}, \tau_{jk}) \in \bR^2$ for $k \not\in \{ i,j\}$.
Using $\la z, J_p(z) \rg = 1$ for $\|z\|_p =1$, we obtain
\begin{align*}
    \la B_{ij}(T) z, J_p(z) \rg &= \tilde{\Psi}(\tau)^{1-p} \Bigl[ |\tau_{ij}|^p + \sum_{k \not\in \{i,j\} } \la z, v^{ij}_k \rg \la J_p(z), J_p(v^{ij}_k) \rg \Bigr] \geq \tilde{\Psi}(\tau)^{1-p} |\tau_{ij}|^p.
\end{align*}
Here, we applied the inequality $\la z,v\rg \la J_p(z), J_p(v) \rg \geq 0$ from~\eqref{eqn:phip-Jp}.
Taking the infimum over $\|z\|_p=1$ gives $\phi_p(B_{ij}(T)) \geq \tilde{\Psi}(\tau)^{1-p} |\tau_{ij}|^p$ for every $i,j$, and hence
\[
\sum_{i<j} \phi_p ( B_{ij}(T)) \geq \tilde{\Psi}(\tau)^{1-p} \sum_{i<j} |\tau_{ij}|^p = \tilde{\Psi}(\tau).
\]
Finally, using $|\tau_T|_2 = 1$, we observe that $\tilde{\Psi}(\tau) = \| \tau_T\|_p \geq c_{N,p} := \binom{N}{2}^{ \min \{ \frac{1}{p} - \frac{1}{2}, 0\} }$.
This establishes the inequality~\eqref{eqn:ell-p-inequality} and completes the proof.
\end{proof}

\begin{remark}\label{rmk:ell-p-general}
By Theorem~\ref{thm:most-general-michael-simon}, establishing the Michael-Simon inequality for the $\ell^p$ anisotropy in every dimension and codimension is reduced to producing a collection of concave barriers $\{ \phi_I \}_{|I|=m}$, associated to every Pl\"ucker $m$-vector, so that the inequality~\eqref{eqn:ell-p-inequality} holds for the projected stress tensors $B_I := E^t_I B_{\Psi}(T) B_I$.
We expect that Theorem~\ref{thm:most-important-application} should indeed extend to the $\ell^p$ norm for every $N,m,p$, with the appropriate choice of barriers.
Our previous choice of concave barrier amounts to $\phi_p(A) := \inf_{ \| z\|_p= 1 } \la Az, J_p(z) \rg$, where $J_p(z) = ( |z_1|^{p-2} z_1, \dots, |z_m|^{p-2} z_m)$ for an $m$-vector $z = (z_1, \dots, z_m)$.
This appears to fail the condition~\eqref{eqn:ell-p-inequality} for $m \geq 3$ and $N,p$ are sufficiently large.
\end{remark}

\bibliography{ref}

\end{document}